\documentclass[psamsfonts]{amsart}  
\usepackage{graphicx}
\usepackage{amssymb}
\newtheorem{theorem}{Theorem}[section]
\newtheorem{lemma}[theorem]{Lemma}
\newtheorem{cor}[theorem]{Corollary}
\theoremstyle{definition}
\newtheorem{remark}{Remark}[section]

\numberwithin{equation}{section}

\usepackage[bookmarks,colorlinks]{hyperref}

\newcommand{\bC}{\mathbb{C}}

\newcommand\E{\mathbb{E}}
\newcommand\R{\mathbb{R}}
\newcommand\e{\varepsilon}

\begin{document}

\title[Submartingale Inequalities]{Burkholder inequalities for submartingales, 
Bessel processes and conformal martingales}\thanks{To appear in American Journal of Mathematics}

\author{Rodrigo Ba\~nuelos}\thanks{R. Ba\~nuelos is supported in part  by NSF Grant
\# 0603701-DMS}
\address{Department of Mathematics, Purdue University, West Lafayette, IN 47907, USA}
\email{banuelos@math.purdue.edu}
\author{Adam Os\c ekowski}\thanks{A. Os\c ekowski is supported in part by MNiSW Grant N N201 364436}
\address{Department of Mathematics, Informatics and Mechanics, University of Warsaw, Banacha 2, 02-097 Warsaw, Poland}
\email{ados@mimuw.edu.pl}

\begin{abstract}
The motivation for this paper comes from the following question on comparison of norms of conformal martingales $X$, $Y$ in $\R^d$, $d\geq 2$. Suppose that $Y$ is differentially subordinate to $X$. For $0<p<\infty$, what is the optimal value of the constant $C_{p,d}$ in the inequality
$$ ||Y||_p\leq C_{p,d}||X||_p ?$$
We answer this question by considering a more general related problem for nonnegative submartingales. This enables us to study extension of the above inequality to the case when $d>1$ is not an integer, which has further interesting applications to stopped Bessel processes and to the behavior of smooth functions on Euclidean domains.  The inequality for conformal martingales, which has its roots on the study of the $L^p$ norms of the Beurling-Ahlfors singular integral operator \cite{BW}, extends a recent result of Borichev, Janakiraman and Volberg \cite{BJV2}.
\end{abstract}

\maketitle

\section{Introduction}
Let $(\Omega,\mathcal{F},\mathbb{P})$ be a complete probability space, filtered by $(\mathcal{F}_t)_{t\geq 0}$, a nondecreasing family of sub-$\sigma$-fields of $\mathcal{F}$. Let $X$, $Y$ be adapted, $\mathbb{R}^d$-valued continuous-path semimartingales. Denote by $[X,X]$ the quadratic variation process of $X$. We refer the reader to Dellacherie and Meyer \cite{DM} for the definition in the one-dimensional case. We set $[X,X]=\sum_{j=1}^d [X^j,X^j]$ in the vector-valued setting where $X^j$ denotes the $j$-th coordinate of $X$. Using the polarization formula we define the quadratic covariance of $X$ and $Y$ by 
$$ [X,Y]=\frac{1}{4}\big([X+Y,X+Y]-[X-Y,X-Y]\big).$$
Following  Ba\~nuelos and Wang \cite{BW} and  Wang \cite{W}, we say that $Y$ is differentially subordinate to $X$ if the process $([X,X]_t-[Y,Y]_t)_{t\geq 0}$ is nondecreasing and nonnegative as a function of $t$. Real-valued semimartingales $X$ and $Y$ are orthogonal if their quadratic covariance $[X,Y]$ has constant trajectories with probability $1$. Finally, we say that $X$ is conformal (or analytic) if for any $1\leq i<j\leq d$, the coordinates $X^i$, $X^j$ are orthogonal and satisfy $[X^i,X^i]=[X^j,X^j]$.
Conformal martingales arise naturally from the composition of analytic functions and Brownian motion in the complex plane and have been studied for many years; see \cite{GetSha} and \cite[p.~177]{RY}. 

If $X$ and $Y$ are martingales, then the differential subordination of $Y$ to $X$ implies many interesting inequalities which have numerous applications in various areas of analysis and  probability.   An excellent source of information in the discrete-time setting is the survey \cite{B1} by Burkholder. One can also find there a detailed description of his method which enables one to obtain sharp versions of such estimates. By an approximation argument and a careful use of It\^o formula, these results can be extended to the continuous-time setting; see the paper by Wang \cite{W}. For other more recent applications of Burkholder's method, the use of his celebrated ``rank-one convex" function and some of its connections to harmonic functions and singular integrals, see \cite{AstIwaPraSak}, \cite{BMH}, \cite{BW0}, \cite{BW}, \cite{B-1.1}, \cite{GME},  \cite{Iwa3}, \cite{JanVol1}, \cite{Jan1},  \cite{NazVol},
 \cite{O1},  \cite{O2},  \cite{O4},  \cite{O5},  and the overview paper \cite{Ban1} which contains extensive list of references on this topic.

Here we recall the celebrated inequality first proved by Burkholder in \cite{B0} in the discrete-time case and extended to the above setting by Wang \cite{W}. Throughout this paper, $||X||_p=\sup_{t\geq 0}||X_t||_p$ for $0<p<\infty$.
\begin{theorem}
Assume that $X$, $Y$ are $\R^d$-valued martingales such that $Y$ is differentially subordinate to $X$. Then
\begin{equation}\label{Burkhin}
 ||Y||_p\leq (p^*-1)||X||_p,\qquad 1<p<\infty,
\end{equation}
where $p^*=\max\{p,p/(p-1)\}$. The constant is the best possible even for $d=1$.
\end{theorem}

The Beurling-Ahlfors operator 
on the complex plane $\bC$  in the singular integral defined by 
\begin{equation}
   \label{eq:ba}
  \mathcal{B}f(z)=\frac{1}{\pi}\, p.v. \int_{\bC}
   \frac{f(w)}{(z-w)^2}\mbox{d}m(w)\,, \quad z\in \bC \ .
\end{equation} 
This operator plays a fundamental role in many areas of analysis and its applications. For some of these connections, we refer the reader to \cite{AstIwaMar}. As a Calder\'on-Zygmund singular integral, $\mathcal{B}$ is bounded on $L^p(\bC)$, for $1<p<\infty$, and the now  celebrated conjecture of T. Iwaniec \cite{Iwa} asserts that  $||\mathcal{B}||_p=p^*-1$.  Burkholder's inequality \eqref{Burkhin} has been crucial in the investigation of Iwaniec's conjecture.   Indeed, the first explicit upper bound  $4(p^*-1)$ for $||\mathcal{B}||_p$ obtained in \cite{BW} used a stochastic integral representation of the operator together with the inequality \eqref{Burkhin}. In addition, the improvement $2(p^*-1)$ obtained by Nazarov and Volberg in \cite{NazVol}, while avoiding the stochastic representation from \cite{BW}, was also based on the inequality \eqref{Burkhin} applied to Haar martingales.   It is observed in \cite[p.~599]{BW} that  in addition to the differential subordination, the martingales arising in the study of the Beurling-Ahlfors operator are in fact conformal martingales and hence, as conjectured in \cite{BW}, one should expect better bounds than the $p^*-1$ of Burkholder. By slightly modifying Burkholder's arguments, the following inequality is established in \cite{BanJan} which takes advantage of the conformality. 

\begin{theorem}
Suppose that $X$, $Y$ are two $\R^d$-valued martingales such that $Y$ is conformal and $\sqrt{\frac{p+d-2}{d(p-1)}}Y$ is differentially subordinate to $X$. Then for $2<p<\infty$,
$$ ||Y||_p\leq (p-1)||X||_p.$$
\end{theorem}
In particular, if $d=2$, $Y$ is conformal and differentially subordinate to $X$, then
\begin{equation}\label{BanJanin}
 ||Y||_p\leq \sqrt{\frac{p(p-1)}{2}}||X||_p,\qquad 2<p<\infty.
\end{equation}

This inequality was used in \cite{BanJan} to prove that $||\mathcal{B}||_p\leq 1.575(p^*-1)$, for $1<p<\infty$, which at this point is the best available bound.  The question immediately arises as to the optimal constant in \eqref{BanJanin}.  Borichev, Janakiraman and Volberg \cite{BJV1}, \cite{BJV2} established the following results in this direction.

\begin{theorem}\label{BJV}
Suppose that $X$ and $Y$ are two $\R^2$-valued martingales on the filtration of 2-dimensional Brownian motion such that $Y$ is differentially subordinate to $X$.
\begin{itemize}

\item[(i)] If $Y$ is conformal, then
$$ ||Y||_p\leq \frac{a_p}{\sqrt{2}(1-a_p)}||X||_p, \qquad 1<p\leq 2,$$
where $a_p$ is the least positive root in the interval $(0, 1)$ of the bounded Laguerre
function $L_p$. This inequality is sharp.

\item[(ii)] If $X$ is conformal,  then
$$ ||Y||_p\leq \frac{\sqrt{2}(1-a_p)}{a_p}||X||_p,\qquad 2\leq p<\infty,$$
where $a_p$ is the least positive root in the interval $(0, 1)$ of the bounded Laguerre
function $L_p$. This inequality is sharp.

\item[(iii)] If $X$ and $Y$ are both conformal, then
\begin{equation}\label{BJVin}
 ||Y||_p\leq \frac{1+z_p}{1-z_p}||X||_p,\qquad 2\leq p<\infty,
\end{equation}
where $z_p$ is the largest root in $[-1, 1]$ of the Legendre function $g$ solving 
$$(1 - s^2)g''(s) - 2sg'(s) + pg(s) = 0.$$
This inequality is sharp. 
\end{itemize}
\end{theorem}

The proof of this theorem, presented in \cite{BJV1} and \cite{BJV2}, is analytic and exploits the Bellman function approach as described in \cite{NazTre1}, \cite{NazTreVol} and \cite{VasVol}.  
The purpose of this paper is to present a significant improvement of the third inequality \eqref{BJVin} which is the main result in \cite{BJV2}. Not only shall we determine the optimal constant in \eqref{BJVin} for the full range $0<p<\infty$, but we will also provide a sharp generalization of this estimate to the $d$-dimensional setting. Since the conformal  two-dimensional martingale treated in \cite{BJV2} are just time-changed $\R^2$-valued Brownian motion, its norm is a time-changed Bessel process in dimension two. This simple observation suggests to study related estimates for stopped Bessel processes. This approach will enable us to investigate the case when the dimension of the Bessel process is an arbitrary number in the interval $(1,\infty)$ and not just an integer. 
 We shall in fact consider an even more general setting. Let $X$, $Y$ be two nonnegative, continuous-path submartingales and let
\begin{equation}\label{DoobMeyer}
X=X_0+M+A,\qquad Y=Y_0+N+B
\end{equation}
be their Doob-Meyer decomposition (see \cite{RY}), uniquely determined by $M_0=A_0=N_0=B_0=0$ and the further condition that $A$, $B$ are predictable. Consider the following property of the finite variation parts of $X$ and $Y$: for a fixed $d>1$ and all  $t> 0$, 
\begin{equation}\label{assumpt}
 X_t\mbox{d}A_t\geq \frac{d-1}{2}\mbox{d}[X,X]_t,\qquad Y_t\mbox{d}B_t\leq \frac{d-1}{2}\mbox{d}[Y,Y]_t.
\end{equation}
For example, if $\overline{X}$, $\overline{Y}$ are  conformal martingales in $\R^d$, then $|\overline{X}|$, $|\overline{Y}|$ are submartingales and by the It\^o formula, their martingale and finite variation parts are
\begin{equation*}
\begin{split}
 M_t&=\sum_{j=1}^d\int_{0+}^t \frac{\overline{X}^j_s}{|\overline{X}_s|}\mbox{d}\overline{X}_s^j,\qquad  A_t=\frac{d-1}{2}\int_{0+}^t \frac{1}{|\overline{X}_s|}\mbox{d}[\overline{X}^1,\overline{X}^1]_s,\\
 N_t&=\sum_{j=1}^d\int_{0+}^t \frac{\overline{Y}^j_s}{|\overline{Y}_s|}\mbox{d}\overline{Y}_s^j,\qquad  B_t=\frac{d-1}{2}\int_{0+}^t \frac{1}{|\overline{Y}_s|}\mbox{d}[\overline{Y}^1,\overline{Y}^1]_s,\qquad t\geq 0. 
\end{split}
\end{equation*}
Hence \eqref{assumpt} is satisfied and in fact, both inequalities become equalities in this case. As another example, if $R$, $S$ are adapted $d$-dimensional Bessel processes and $\tau$ is a stopping time, then $X=(R_{\tau\wedge t})_{t\geq 0}$, $Y=(S_{\tau\wedge t})_{t\geq 0}$ enjoy the property \eqref{assumpt}.

We now turn to a precise statement of our main result. For a given $0<p<\infty$ and $d>1$ such that $p+d>2$, let $z_0=z_0(p,d)$ be the smallest root in $[-1, 1)$ of the solution to \eqref{diff} (see \S2 below) and let
\begin{equation}\label{defcp}
C_{p,d}=\begin{cases}
\frac{1+z_0}{1-z_0}, & \mbox{if }(2-d)_+<p\leq 2,\\\\
\frac{1-z_0}{1+z_0}, &  \mbox{if }2<p<\infty.
\end{cases}
\end{equation}
\begin{theorem}\label{mainth}
Let $X$, $Y$ be two nonnegative submartingales satisfying \eqref{assumpt} and such that $Y$ is differentially subordinate to $X$. Then for $(2-d)_+ <p<\infty$ we have
\begin{equation}\label{mainin}
 ||Y||_p\leq C_{p,d} ||X||_p
\end{equation}
and the constant $C_{p,d}$ is the best possible. If $0<p\leq (2-d)_+$, then the moment inequality does not hold with any finite $C_{p,d}$.
\end{theorem}

As an application, we have the following bound for conformal martingales and Bessel processes.  The first result extends the Borichev--Janakiraman--Volberg result (iii) in Theorem \ref{BJV} (see Remark \ref{constant} below).

\begin{cor}\label{conf-cor}
Assume that $X$, $Y$ are conformal martingales in $\R^d$, $d\geq 2$, such that $Y$ is differentially subordinate to $X$. Then for any $0<p<\infty$,
\begin{equation}\label{orthogin}
 ||Y||_p\leq C_{p,d}||X||_p
\end{equation}
and the constant $C_{p,d}$ is the best possible. 
\end{cor}

\begin{cor}
Assume that $R$, $S$ are $d$-dimensional Bessel processes, $d>1$, driven by the same Brownian motion and satisfying \eqref{assumpt}. Then for any $(2-d)_+ <p<\infty$  and any stopping time $\tau\in L^{p/2}$, we have
\begin{equation}\label{besselin}
 ||S_{\tau}||_p\leq C_{p,d}||R_\tau||_p
\end{equation}
and the constant $C_{p,d}$ is the best possible. If $0<p\leq (2-d)_+$, then the moment inequality does not hold with any finite $C_{p,d}$.
\end{cor}

The paper is organized as follows. In \S2, we introduce a differential equation which is closely associated with these inequalities and study its solutions satisfying certain boundedness property. These solutions are then exploited in \S3 in the construction of special functions, which, by the use of Burkholder's method, yield the assertion of Theorem \ref{mainth}. The final part of the paper is devoted to applications of our results to harmonic functions on Euclidean domains.

\section{A differential equation}
Throughout this section, $0<p<\infty$ and $d>1$ are given and fixed. We emphasize that $d$ need not be an integer.
We start with some preliminary facts and properties of $d$-dimensional Bessel processes. 
Let $B$ be a standard one-dimensional Brownian motion and let $R$, $S$ be two Bessel processes of dimension $d$, satisfying the stochastic differential equations
\begin{equation}\label{bessels}
\begin{split}
\mbox{d}R_t&=\mbox{d}B_t+\frac{d-1}{2}\frac{\mbox{d}t}{R_t},\\
\mbox{d}S_t&=-\mbox{d}B_t+\frac{d-1}{2}\frac{\mbox{d}t}{S_t}\\
\end{split}
\end{equation}
for all $t\geq 0$. As already mentioned in the Introduction, these processes, if stopped appropriately, are the extremals in \eqref{mainin} and hence are strictly related to the structure of our problem.  We refer the reader to \cite{RY} for some of the basic properties of Bessel processes including their stochastic differential equation representation given above. 

Let us recall some basic inequalities, which will be needed in our subsequent considerations. Assume that $R$ starts from $x\geq 0$. The Burkholder--Gundy inequalities for Bessel processes proved by DeBlassie \cite{Deb} states that there are constants $c_{p,d}$, $c_{p,d}'$, depending only on the parameters indicated, such that 
\begin{equation}\label{BDG}
 c_{p,d}||(x^2+\tau)^{1/2}||_p\leq ||R_{\tau}^*||_p\leq c_{p,d}'||(x^2+\tau)^{1/2}||_p,
\end{equation}
for any stopping time $\tau$. Here, as usual, $R^*$ denotes the maximal process of $R$, given by $R^*_t=\sup_{0\leq s\leq t}R_s$. Another important result is Doob's maximal inequality for Bessel processes. This states that  if $p+d>2$, then there is $c_{p,d}''$ depending only on $p$ and $d$ such that
\begin{equation}\label{Doob}
 ||R_{\tau}^*||_p\leq c_{p,d}''||R_{\tau}||_p
\end{equation}
for all stopping times $\tau$ which are $p/2$-integrable. We refer to Pedersen \cite{Ped} where this inequality is obtained with the best constant. 

Let us turn to the differential equation which plays a fundamental role in the paper:
\begin{equation}\label{diff}
(1-s^2)g''(s)-2(d-1)sg'(s)+p(d-1)g(s)=0.
\end{equation}
We shall prove now that there is a continuous function $g=g_{p,d}:[-1,1)\to \R$ with $g(-1)=-1$, satisfying \eqref{diff} for $s\in (-1,1)$ and hence bounded on any compact subinterval of $[-1,1)$. Consider the class of power series of the form
\begin{equation}\label{defg}
 g(s)=\sum_{n=0}^\infty a_n(1+s)^n,
\end{equation}
with $a_0=-1$. 
Plugging this into \eqref{diff} and comparing the coefficients of $(1+s)^n$, we obtain
\begin{equation}\label{system}
\begin{split}
a_{n+1}=-\prod_{k=0}^n \frac{k(k-1)+2(d-1)k-p(d-1)}{2(k+1)(k+d-1)}, \qquad \mbox{for }n\geq 0.
\end{split}
\end{equation}
It is easy to see that $\lim_{n\to\infty}|a_n|^{1/n}=1/2$, so the radius of convergence of the series for $g$ is indeed $2$ and hence \eqref{defg} gives the function we are looking for. Throughout the paper, $z_0=z_0(p,d)$ denotes the smallest root of the solution $g_{p,d}$  (if $g_{p,d}$ has no zeros, put $z_0=1$).

The differential equation \eqref{diff} arises as follows. For $x>0$ and $y\geq 0$, let
\begin{equation}\label{defW}
W(x,y)=(x+y)^pg\left(\frac{y-x}{x+y}\right).
\end{equation}
We have that $W$ is of class $C^\infty$ on $(0,\infty)\times (0,\infty)$. In fact, since $g$ is well defined on $(-3,1)$, we see that the partial derivatives of $W$ can be extended to continuous functions on $(0,\infty)\times [0,\infty)$. 
Fix $a\in (-1,1)$ and introduce the stopping time
\begin{equation}\label{deftau}
 \tau^a=\inf\left\{t\geq 0: S_t\geq \frac{1+a}{1-a}R_t\right\}.
\end{equation}

\begin{lemma}\label{optional}
Let $R,\,S$ be Bessel processes as in \eqref{bessels}, starting from $x,\,y>0$, respectively. Then for any $a\in (-1,1)$, the process $(W(R_{\tau^a\wedge t},S_{\tau^a\wedge t}))_{t\geq 0}$ is a martingale.
\end{lemma}
\begin{proof}
Of course, we may assume that $y<\frac{1+a}{1-a}x$, since otherwise $\tau^a\equiv 0$ and the claim is trivial. 
The situation is easy when $d\geq 2$. Since $0$ is polar for $R$ and $S$, we may apply It\^o formula and we check that \eqref{diff} implies that the finite variation part of $(W(R_{\tau^a\wedge t},S_{\tau^a\wedge t}))_{t\geq 0}$ vanishes. The latter amounts to saying that
\begin{equation}\label{finvar}
\frac{d-1}{2x}W_x(x,y)+\frac{d-1}{2y}W_y(x,y)+\frac{1}{2}\left[W_{xx}(x,y)-2W_{xy}(x,y)+W_{yy}(x,y)\right]=0
\end{equation}
for all $x,\,y>0$. 
 For $d<2$, the situation is more complicated, since $S$ reaches $0$ with probability $1$; on the other hand, there are no problems with $R$: $R>0$ almost surely on $[0,\tau^a]$. We shall prove the claim by checking that $ \E W(R_{\sigma},S_{\sigma})=W(x,y)$ 
for any bounded stopping time $\sigma$ such that $\sigma\leq \tau^a$ almost surely. To do this, we use standard approximation procedure and work with the squares of $R$ and $S$, which satisfy the stochastic differential equations
$$ \mbox{d}R_t^2=2R_t\mbox{d}B_t+d\mbox{d}t,\qquad \mbox{d}S_t^2=-2S_t\mbox{d}B_t+d\mbox{d}t \qquad \mbox{for } t\geq 0.$$
Let $N,\,\e$ be positive numbers and put $\eta=\inf\{t\geq 0:R_t+S_t\geq N\}.$ Define $\overline{W}(u,v)=W(u^{1/2},(\e+v)^{1/2})$ for $u,\,v\geq 0$. This function has the necessary smoothness and we may apply It\^o formula to obtain
\begin{equation}\label{ito_formula}
 \E \overline{W}(R^2_{\sigma \wedge \eta},S^2_{\sigma \wedge \eta})= \overline{W}(x^2,y^2)+\E \int_{0+}^{\sigma\wedge \eta} \mathcal{L}\overline{W}(R^2_s,S^2_s)\mbox{d}s,
\end{equation}
where
\begin{equation*}
\begin{split}
 \mathcal{L}&\overline{W}(u,v)\\
&=\overline{W}_x(u,v)d+\overline{W}_y(u,v)d+2u\overline{W}_{xx}(u,v)-4(uv)^{1/2}\overline{W}_{xy}(u,v)+2v\overline{W}_{yy}(u,v)\\
&=\frac{d-1}{2u^{1/2}}W_x(u^{1/2},(\e+v)^{1/2})+\frac{d}{2(\e+v)^{1/2}}W_y(u^{1/2},(\e+v)^{1/2})\\
&\quad +\frac{1}{2}W_{xx}(u^{1/2},(\e+v)^{1/2})-\frac{v^{1/2}}{(\e+v)^{1/2}}W_{xy}(u^{1/2},(\e+v)^{1/2})\\
&\quad +\frac{v}{2(\e+v)}W_{yy}(u^{1/2},(\e+v)^{1/2})-\frac{v}{2(\e+v)^{3/2}}W_y(u^{1/2},(\e+v)^{1/2}).
\end{split}
\end{equation*}
Applying \eqref{finvar} and calculating a little bit, we get
\begin{equation*}
\begin{split}
 \mathcal{L}\overline{W}(u,v)&=\frac{\e}{2(\e+v)}\left(\frac{W_y(u^{1/2},(\e+v)^{1/2})}{(\e+v)^{1/2}}-W_{yy}(u^{1/2},(\e+v)^{1/2})\right)\\
&\quad +\left[1-\left(\frac{v}{\e+v}\right)^{1/2}\right]W_{xy}(u^{1/2},(\e+v)^{1/2}).
\end{split}
\end{equation*}
Now, if $\e\to 0$, then each of the two summands on the right converges to $0$ uniformly on the set $F=\{(u,v):x+y\leq u^{1/2}+v^{1/2}\leq N,\,v^{1/2}\leq \frac{1+a}{1-a}u^{1/2}\}$. This in an immediate consequence of the equalities  $W_y(x,0)=0$ and $W_{xy}(x,0)=0$ valid for all $x>0$. However, the process $((R_{\sigma\wedge \eta\wedge t}^2,S^2_{\sigma\wedge \eta\wedge t}))_{t\geq 0}$ takes values in $F$ if $N$ is sufficiently large; this follows from the bound $y<\frac{1+a}{1-a}x$ (which we have assumed at the beginning of the proof) and the fact that the process $R+S$ is nondecreasing (see \eqref{bessels}).   
Hence, by Lebesgue's dominated convergence theorem, \eqref{ito_formula} yields
$$ \E W(R_{\sigma\wedge \eta},S_{\sigma\wedge \eta})=W(x,y).$$
Now we let $N$ go to $\infty$ and the claim follows, again by Lebesgue's dominated convergence theorem. To see this, note that
$$ |W(R_{\sigma\wedge \eta},S_{\sigma\wedge \eta})|\leq \sup_{[-1,a]}|g|\cdot (R_{\sigma}^*+S_{\sigma}^*)^p$$
and observe that the right-hand side is integrable, in virtue of \eqref{BDG} and the boundedness of $\sigma$.
\end{proof}

\begin{lemma}\label{roots}
We have $z_0<1$ if and only if $p+d>2$.
\end{lemma}
\begin{proof} 
Let $p+d>2$ and assume that $g$ has no roots smaller than $1$. Let $R$, $S$ be Bessel processes as in \eqref{bessels}, starting from $x$, $y>0$. Suppose that $\tau$ is a stopping time satisfying $\E \tau^{p/2}<\infty$. By \eqref{BDG} and \eqref{Doob}, there are constants $c_1$, $c_2$, $c_3$, depending only on $x$, $y$, such that
\begin{equation}\label{dep}
 ||S_\tau||_p\leq c_1||(y^2+\tau)^{1/2}||_p\leq c_2||R_{\tau}^*||_p\leq c_3||R_\tau||_p.
\end{equation}
Recall the stopping time $\tau^a$ given by \eqref{deftau}. If $y<\frac{1+a}{1-a}x$, then $\tau^a>0$ almost surely and by Lemma \ref{optional},
\begin{equation*}
\begin{split}
W(x,y)&=\E W(R_{\tau^a\wedge t},S_{\tau^a\wedge t})\leq \sup_{[-1,a]}g\cdot\E (R_{\tau^a\wedge t}+S_{\tau^a\wedge t})^p.
\end{split}
\end{equation*}
Since $g$ has no roots in $(-1,1)$, the number $\sup_{[-1,a]}g$ is negative and hence we may write
$$ W(x,y)\leq \sup_{[-1,a]}g \cdot \E R_{\tau^a\wedge t}^p,$$
or, equivalently,
\begin{equation}\label{fyl}
 \E R_{\tau^a\wedge t}^p\leq W(x,y)(\sup_{[-1,a]}g)^{-1}.
\end{equation}
By \eqref{BDG} and \eqref{Doob}, this implies that $\tau^a$ is $p/2$-integrable. Moreover, directly from the definition of $\tau^a$,
\begin{equation}\label{sharpness}
 ||S_{\tau^a}||_p=\frac{1+a}{1-a}||R_{\tau^a}||_p,
\end{equation}
which contradicts \eqref{dep} if $a$ is sufficiently close to $1$. Thus, $g$ must have a root inside the interval $(-1,1)$.

To get the reverse implication, note first that if $p+d=2$, then $g_{p,d}(s)=\left(\frac{1-s}{2}\right)^p$, which does not have roots smaller than $1$. Furthermore, the reasoning presented above shows that $\tau^a\in L^{p/2}$ for any $a<1$ and any starting points $x$, $y$. Next, suppose that $p+d<2$, assume that $g_{p,d}$ has at least one zero smaller than $1$ and let $a$ stand for the smallest root. Suppose that the starting points $x$, $y$ satisfy $y<\frac{1+a}{1-a}x$. As we have just observed, $\tau^a\in L^{(2-d)/2}$, which in view of \eqref{BDG} yields \begin{equation}\label{r*}
R_{\tau^a}^*\in L^{2-d}.
\end{equation} 
By Lemma \ref{optional},
$$ W(x,y)=\E W(R_{\tau^a\wedge t},S_{\tau^a\wedge t})=\E W(R_t,S_t)1_{\{\tau^a> t\}},$$
because $W(R_{\tau^a},S_{\tau^a})=(R_{\tau^a}+S_{\tau^a})^pg(a)=0$. However, the expression on the right hand side converges to zero as $t\to\infty$. Indeed,
\begin{equation*}
\begin{split}
|\E W(R_t,S_t)1_{\{\tau^a>t\}}|&\leq \sup_{[-1,a]}|g|\E(R_t+S_t)^p1_{\{\tau^a>t\}}\\
&\leq \sup_{[-1,a]}|g|\left(\frac{2}{1-a}\right)^p\E R_t^p1_{\{\tau^a>t\}},
\end{split}
\end{equation*}
where in the latter passage we have used the definition of $\tau^a$.
By  Lebesgue's dominated convergence theorem and \eqref{r*}, letting $t\to \infty$ yields $W(x,y)=0$ and hence $a$ is not the smallest root of $g$. The obtained contradiction completes the proof.
\end{proof}

\begin{remark}\label{constant}
Before we proceed, let us assure the reader that $C_{p,2}$ and the constant in \eqref{BJVin} coincide, though the latter involves the \emph{largest} root $z_p$ of a solution to \eqref{diff}. The reason for this is that Borichev, Janakiraman and Volberg work with the reflected function $s\mapsto g_{p,2}(-s)$, which also solves \eqref{diff}; thus $z_p=-z_0$ and $\frac{1+z_p}{1-z_p}=C_{p,2}$.
\end{remark}

In the remainder of this section we investigate several other  properties of the function $g$ which will be useful later.  Such technical properties are always part of these type of optimal constant problems.  Different (but in the same spirit) technical results are also derived in \cite{BJV1} and \cite{BJV2}. 

\begin{lemma}\label{lem1}
The function $g$ enjoys the following.
\begin{itemize}
\item[(i)] We have $g'(s)>0$ for $s\in (-1,z_0]$.

\item[(ii)] If $p\leq 2$, then $g$ is convex on $[-1,z_0)$. If $p\geq 2$, then $g$ is concave on $[-1,z_0)$. If $p\neq 2$, then the convexity/concavity is strict.

\item[(iii)] We have $z_0> 0$ for $p< 2$, $z_0=0$ for $p=2$, and $z_0< 0$ for $p> 2$.
\end{itemize}
\end{lemma}
\begin{proof}
(i) Observe that $g'(z_0)=0$ is impossible: then by \eqref{diff} and straightforward induction we would have $g^{(n)}(z_0)=0$ for all $n\geq 0$, which would further imply that $g$ is identically $0$, as an analytic function. Consequently, all we need is to verify the inequality $g'>0$ on the open interval $(-1,z_0)$. The function $g$ is strictly increasing in a neighborhood of $-1$, since $\lim_{s\downarrow -1}g'(s)=a_1=p/2$. 
Suppose that the set $\{s< z_0:g'(s)=0\}$ is nonempty and let $s_0$ denote its infimum. Then 
$s_0\in (-1,z_0)$, $g'(s_0)=0$ and $g'(s)>0$ for $s<s_0$. This gives $g''(s_0)\leq 0$, which combined with \eqref{diff} implies $g(s_0)\geq 0$, a contradiction.

(ii) The case $p=2$ is trivial, since then $g(s)=s$ for all $s\in [-1,1]$; thus we may and do assume that $p\neq 2$. We shall prove that $(2-p)g$ is strictly convex on $[-1,z_0]$, using essentially the same argument as in (i). We have that $(2-p)g''$ is positive in the neighborhood of $-1$, since, by \eqref{system},
$$ \lim_{s\downarrow -1}g''(s)=2a_2=\frac{p(2-p)(d-1)}{4d}.$$
Next, assume that the set $\{s<z_0: (2-p)g''(s)=0\}$ is nonempty and denote its infimum by $s_0$. Then $s_0\in (-1,z_0)$, $(2-p)g''(s_0)=0$ and $(2-p)g''(s)>0$ for $s\in (-1,s_0)$, which in particular implies $(2-p)g'''(s_0)\leq 0$. Differentiating \eqref{diff} and applying the latter inequality yields
\begin{equation*}
\begin{split}
0&\geq (1-s_0^2)(2-p)g'''(s_0)\\
&=2(2-p)ds_0g''(s_0)+(2-p)^2(d-1)g'(s_0)=(2-p)^2(d-1)g'(s_0),
\end{split}
\end{equation*}
which contradicts (i).

(iii) As previously, the case $p=2$ is trivial (we have $g(s)=s$ for all $s$). If $p\leq (2-d)_+$, then $z_0=1$. If $(2-d)_+<p<2$, then using \eqref{diff} and (ii),
$$ 0=(1-z_0^2)g''(z_0)-2(d-1)z_0g'(z_0)\geq -2(d-1)z_0g'(z_0).$$
Consequently, if the assertion was not true, we would get $g'(z_0)\leq 0$. By (i), the mean value theorem would imply that  $g''$ is negative at some point in the interval $(-1,z_0)$.  However, this is impossible in view of (ii). If $p>2$, then substituting $s=0$ into \eqref{diff} gives $g''(0)+p(d-1)g(0)=0$. Now $z_0>0$  would imply $g(0)<0$ and $g''(0)>0$, which has been excluded this in (ii). On the other hand, $z_0=0$ also leads to a contradiction. Indeed, it yields $g''(0)=0$ and hence $g'''(0)\geq 0$, in view of (ii). However, differentiating \eqref{diff} gives  $g'''(0)=(2-p)(d-1)g'(0)<0$.
\end{proof}

For any $p>0$, we introduce the function $v=v_p:[-1,1]\to \R$ defined by 
$$ v(s)=\left(\frac{1+s}{2}\right)^p-\left(\frac{1+z_0}{1-z_0}\right)^p\left(\frac{1-s}{2}\right)^p.$$
We have 
\begin{equation*}
\begin{split}
 v''(s)=\frac{p(p-1)}{2^p}\left[(1+s)^{p-2}-\left(\frac{1+z_0}{1-z_0}\right)^p(1-s)^{p-2}\right].
 \end{split}
 \end{equation*}
For $p\neq 2$, let $s_1=s_1(p)$ denote the unique root of the expression in the square brackets above. It is easy to verify that $s_1<0$ and $s_1<z_0$, using Lemma \ref{lem1} (iii). For $p\geq 1$, 
let $c=c(p)$ be the unique positive constant for which $cg'(z_0)=v'(z_0)$.  A calculation gives
$$ c=\frac{2p(1+z_0)^{p-1}}{2^pg'(z_0)(1-z_0)}.$$

\begin{lemma}
(i) Let $1\leq p\leq 2$. Then for $s\in [-1,z_0]$ we have
\begin{equation}\label{maj<2}
cg(s)\geq v(s).
\end{equation}

(ii) Let $p\geq 2$. Then for $s\in [-1,z_0]$ we have
\begin{equation}\label{maj>2}
cg(s)\leq v(s).
\end{equation}
\end{lemma}
\begin{proof}
For $p=2$ we have $cg(s)=v(s)$, so both \eqref{maj<2} and \eqref{maj>2} hold true; hence we may assume that $p\neq 2$. We treat (i) and (ii) in a unified manner and show that $$c(2-p)g(s)\geq (2-p)v(s)$$
for $s\in [-1,z_0]$.  
We have that $(2-p)v''(s)\geq 0$ for $s\in (-1,s_1)$ and $(2-p)v''(s)\leq 0$ for $s\in (s_1,1)$. Since $(2-p)g$ is a strictly convex function, we see that \eqref{maj<2} holds on $[s_1,z_0]$ and is strict on $[s_1,z_0)$. 
Suppose that the set $\{s<z_0:cg(s)=v(s)\}$ is nonempty and let $s_0$ denote its supremum. Then $s_0<s_1$, $cg(s_0)=v(s_0)$ and $(2-p)cg(s)>(2-p)v(s)$ for $s\in (s_0,z_0)$, which implies $(2-p)cg'(s_0)\geq (2-p)v'(s_0)$. In consequence, by \eqref{diff},
\begin{equation*}
\begin{split}
 0&< (1-s_0^2)(2-p)cg''(s_0)\\
&=(d-1)(2-p)(2s_0cg'(s_0)-pg(s_0))\\
 &\leq (d-1)(2-p)(2s_0v'(s_0)-pv(s_0))\\
 &=-\frac{p(2-p)(d-1)(1-s_0^2)}{2^p}\left[(1+s_0)^{p-2}-\left(\frac{1+z_0}{1-z_0}\right)^p(1-s_0)^{p-2}\right].
\end{split}
\end{equation*}
This yields $s_0\geq s_1$ (see the definition of $s_1$), a contradiction.
\end{proof}

The inequality \eqref{maj<2} is also valid for $p\in ((2-d)_+,1)$, but this seems to be more difficult. To overcome this problem, fix such a $p$ and consider the set
$$ \{\alpha\geq 0: \alpha g(s)\geq v(s)\quad \mbox{for all }s\in [-1,z_0]\}.$$
Of course, this set is a closed, bounded subinterval of $\R_+$ and contains $0$. In fact, it has a nonempty interior, since $v$ is strictly increasing, $v'(z_0)>0$ and $g$ is a convex function. Define $c=c(p)$ as the right endpoint of this interval. Then, obviously, we have
\begin{equation}\label{maj<1}
cg(s)\geq v(s),\qquad \mbox{for }s\in [-1,z_0],
\end{equation}
and we can show the following.

\begin{lemma}
There exists $z_1=z_1(p)\in (s_1,z_0]$ for which 
\begin{equation}\label{match}
cg(z_1)=v(z_1),\qquad  cg'(z_1)=v'(z_1)
\end{equation}
and
\begin{equation}\label{boundv''}
v''(s)\geq 0, \qquad \mbox{for }s\geq z_1.
\end{equation}
\end{lemma}
\begin{proof}
We have $cg(z_0)=v(z_0)$, so \eqref{maj<1} implies $cg'(z_0)\leq v'(z_0)$, or
\begin{equation}\label{boundc}
 c\leq \frac{2p(1+z_0)^{p-1}}{2^pg'(z_0)(1-z_0)}.
\end{equation}
If we have equality here, we can take $z_1=z_0$. Then \eqref{match} is obviously satisfied and the validity of \eqref{boundv''} follows from
$$ v''(s)\geq v''(z_0)=-\frac{p(p-1)(1+z_0)^{p-2}z_0}{2^{p-2}(1-z_0)^2}> 0.$$
Suppose that the inequality in \eqref{boundc} is strict: $cg'(z_0)<v'(z_0)$. Then the set
$$ \{z<z_0: cg(z)=v(z)\}$$
is nonempty (if it was not, we would be able to increase $c$ a little bit and \eqref{maj<1} would still hold). Let $z_1$ denote the infimum of this set. Then $z_1>-1$ and it is clear that \eqref{match} holds true, as well as the bound $cg''(z_1)\geq v''(z_1)$. By virtue of \eqref{diff}, we get
\begin{equation*}
\begin{split}
0&\geq (1-s^2)v''(z_1)-2(d-1)z_1v(z_1)+p(d-1)v(z_1)\\
&=\frac{p(p+d-2)(1-z_1^2)}{2^p}\left[(1+z_1)^{p-2}-\left(\frac{1+z_0}{1-z_0}\right)^p(1-z_1)^{p-2}\right]\\
&=\frac{(p+d-2)(1-z_1^2)}{p-1}v''(z_1). 
\end{split}
\end{equation*}
This gives \eqref{boundv''}, since $v''$ is nondecreasing.
\end{proof}

The next properties of $g$ we will need are gathered in the following.
\begin{lemma} Assume that $p+d>2$ and $s\in (-1,z_0]$.
\begin{itemize}
\item[(i)] We have
\begin{equation}\label{diffin1}
 (2-p)(1-s^2)g''(s)-2(p-1)(p-2)sg'(s)+p(p-1)(p-2)g(s)\geq 0.
\end{equation}

\item[(ii)] We have
\begin{equation}\label{diffin2}
 s(1-s^2)g''(s)-[p+d-2+(d-p)s^2]g'(s)+p(d-1)sg(s)\leq 0.
\end{equation}
\end{itemize}
\end{lemma}
\begin{proof}
By \eqref{diff}, the inequality \eqref{diffin1} can be rewritten in the form
$$ \frac{(p+d-2)(2-p)(1-s^2)g''(s)}{d-1}\geq 0,$$
while \eqref{diffin2} is equivalent to 
$$ -(p+d-2)(1-s^2)g'(s)\leq 0.$$
Both these estimates follow at once from Lemma \ref{lem1}. 
\end{proof}

\begin{lemma} Let $s\in (-1,z_0]$. 
\begin{itemize}
\item[(i)] If $(2-d)_+<p\leq 2$, then
\begin{equation}\label{diffin3}
pg(s)+(1-s)g'(s)\geq 0.
\end{equation}

\item[(ii)] If $p\geq 2$, then
\begin{equation}\label{diffin4}
pg(s)-(1+s)g'(s)\leq 0.
\end{equation}
\end{itemize}
\end{lemma}
\begin{proof} The second statement is trivial, since $g(s)\leq 0$ and $g'(s)\geq 0$. To show (i), note that both sides become equal when we let $s\to -1$ and
$$ \lim_{s\downarrow -1}(pg(s)+(1-s)g'(s))'=\lim_{s\downarrow-1} \big[(p-1)g'(s)+(1-s)g''(s)\big]=\frac{p(p+d-2)}{2d}>0.$$
Therefore, the inequality holds in neighborhood of $-1$. Now, suppose that the set $\{s\leq z_0: pg(s)+(1-s)g'(s)<0\}$ is nonempty and let $s_0$ denote its infimum. Then $s_0\leq z_0$,  $pg(s_0)+(1-s_0)g'(s_0)=0$ and $(p-1)g'(s_0)+(1-s_0)g''(s_0)\leq 0$. Using \eqref{diff}, these the latter two statements yield $ (p+d-2)(1+s_0)g'(s_0)\leq 0,$  
a contradiction with Lemma \ref{lem1} (i).
\end{proof}

\section{Proof of Theorem \ref{mainth}}

Throughout this section, we assume that $0<p<\infty$, $d>1$ are fixed and satisfy $p+d>2$. Recall the numbers $c=c(p)$, $z_0=z_0(p)$ and $z_1=z_1(p)$ introduced in the previous section. 
We start by defining special functions $U=U_{p,d}:\R_+^2\to \R$. For $p<1$, let
$$ U_{p,d}(x,y)=\begin{cases}
c(x+y)^p g_{p,d}\left(\frac{y-x}{x+y}\right) & \mbox{if }y\leq \frac{1+z_1}{1-z_1}x,\\
y^p-C_{p,d}^px^p & \mbox{if }y>\frac{1+z_1}{1-z_1}x
\end{cases}$$
and for $1\leq p\leq 2$,
$$ U_{p,d}(x,y)=\begin{cases}
c(x+y)^p g_{p,d}\left(\frac{y-x}{x+y}\right) & \mbox{if }y\leq \frac{1+z_0}{1-z_0}x,\\
y^p-C_{p,d}^px^p & \mbox{if }y>\frac{1+z_0}{1-z_0}x.
\end{cases}$$
For $p>2$, the formula is slightly different:
$$ U_{p,d}(x,y)=\begin{cases}
-cC_{p,d}^p(x+y)^p g_{p,d}\left(\frac{x-y}{x+y}\right) & \mbox{if }y\geq \frac{1-z_0}{1+z_0}x,\\
y^p-C_{p,d}^px^p & \mbox{if }y<\frac{1-z_0}{1+z_0}x.
\end{cases}$$
Moreover, let $ V_{p,d}(x,y)=y^p-C_{p,d}^px^p$ for any $p$. We will skip the lower indices and write $U$, $V$ instead of $U_{p,d}$ and $V_{p,d}$ as doing so produces no risk of ambiguity. Let
$$ L(x,y)=U_{xx}(x,y)+\frac{(d-1)U_{x}(x,y)}{x},\qquad R(x,y)=U_{yy}(x,y)+\frac{(d-1)U_{y}(x,y)}{y}.$$
We shall need the following facts.
 
 \begin{lemma}\label{propUV}
We have 
 \begin{equation}\label{maj}
 \qquad \qquad U(x,y)\geq V(x,y),
 \end{equation}
 \begin{equation}\label{part1}
 L(x,y)+R(x,y)-2U_{xy}(x,y)\leq 0, \qquad \qquad \qquad 
 \end{equation}
 \begin{equation}\label{part1.5}
 L(x,y)-R(x,y)\leq 0,\qquad \,\,\, 
 \end{equation}
 \begin{equation}\label{part2}
 \quad U_{xy}(x,y)\leq 0,
 \end{equation}
 \begin{equation}\label{part3}
 \qquad \,\,U_{x}(x,y)\leq 0,\qquad  U_{y}(x,y)\geq 0,
 \end{equation}
 for all $(x,y)$ at which the involved partial derivatives of $U$ exist.
 \end{lemma}
 \begin{proof}
In fact, the nontrivial parts of these estimates have been already established in the previous section. For example, suppose that $p<1$. If $y<\frac{1+z_1}{1-z_1}x$, then \eqref{maj} is equivalent to \eqref{maj<1}, both sides of \eqref{part1} are equal (we obtain \eqref{diff}, actually), \eqref{part1.5} reduces to \eqref{diffin1}, \eqref{part2} follows from \eqref{diffin2} and, finally, \eqref{part3} is a consequence of \eqref{diffin3} and \eqref{diffin4}. Suppose then, that $y>\frac{1+z_1}{1-z_1}x$. Then both sides of \eqref{maj} are equal and, since $C_{p,d}\geq 1$, 
\begin{eqnarray*} 
L(x,y)+R(x,y)-2U_{xy}(x,y)&=&p^2y^{p-2}-p^2C_{p,d}^px^{p-2}\\
&\leq& p^2C_{p,d}^{p-2}(1-C_{p,d}^2)x^{p-2}\\
&\leq& 0,
\end{eqnarray*}
so \eqref{part1} is satisfied. Since $L(x,y)\leq 0$ and $R(x,y)\geq 0$, \eqref{part1.5} holds as well. We have $U_{xy}=0$, which gives \eqref{part2}. Finally, \eqref{part3} is trivial. 
The remaining cases $1\leq p\leq 2$ and $p>2$ are verified essentially in the same manner. We leave the details to the reader.
\end{proof}

The proof of the inequality \eqref{mainin} will be based on It\^o formula. However, since $U$ is not of class $C^2$ (at least when $p\neq 2$), we are forced to modify it slightly to ensure the necessary smoothness. To accomplish this we use the ``mollification" trick first employed by Burkholder in \cite{Bur43} and subsequently by Wang in \cite{W}, and others. Consider a $C^\infty$ function $\psi:\R^2\to [0,\infty)$, supported on a ball centered at $0$ and radius $1$, satisfying $\int_{\R^2}\psi=1$. Fix $\delta>0$ and define $U^{\delta},V^{\delta}:[2\delta,\infty)\times [2\delta,\infty)\to\R$ by
\begin{equation*}
\begin{split}
 U^{\delta}(x,y)&=\int_{[-1,1]^2} U(x+\delta-\delta u,y-\delta-\delta v)\psi(u,v)\mbox{d}u\mbox{d}v,\\
 V^{\delta}(x,y)&=\int_{[-1,1]^2} V(x+\delta-\delta u,y-\delta-\delta v)\psi(u,v)\mbox{d}u\mbox{d}v
 \end{split}
 \end{equation*}
 (note that we add $\delta$ on the first coordinate and subtract $\delta$ on the second).  
 The key property of $U^\delta$ is the following.
 
 \begin{lemma}
 For any $x,\,y>2\delta$ and $h,\,k\in \R$ we have
\begin{eqnarray}\label{conv}
 \left[U_{xx}^\delta(x,y)+\frac{(d-1)U_{x}^\delta(x,y)}{x}\right]h^2 &+&2U_{xy}^\delta(x,y)hk \nonumber\\
& +&\left[U_{yy}^\delta(x,y)+\frac{(d-1)U_{y}^\delta(x,y)}{y}\right]k^2\nonumber \\\nonumber\\
 &\leq & w(x,y)\cdot(h^2-k^2),
\end{eqnarray}
where
$$ w(x,y)=\frac{1}{2}\int_{[-1,1]^2} (L-R)(x+\delta-\delta u,y-\delta-\delta v)\psi(u,v) \mbox{d}u\mbox{d}v\leq 0.$$
 \end{lemma}
\begin{proof}
Since $U$ is of class $C^1$,  integration by parts yields
\begin{equation*}
\begin{split}
 U^{\delta}_{x}(x,y)&=\int_{[-1,1]^2} U_{x}(x+\delta-\delta u,y-\delta-\delta v)\psi(u,v)\mbox{d}u\mbox{d}v
\end{split}
\end{equation*}
for all $x,\,y>2\delta$, 
and similar identities hold for $U^{\delta}_{y}$, $U^{\delta}_{xx}$, $U^{\delta}_{xy}$ and $U^{\delta}_{yy}$. Let $h,\,k$ be two real numbers. By \eqref{part1} and \eqref{part2}, $L+R$ is nonpositive and
\begin{equation}\label{aux1}
\begin{split}
|2U_{xy}^{\delta}(x,y)&hk|\\
&\leq -|hk|\int_{[-1,1]^2} (L+R)(x+\delta-\delta u,y-\delta-\delta v)\psi(u,v) \mbox{d}u\mbox{d}v\\
&\leq -\frac{h^2+k^2}2\int_{[-1,1]^2} (L+R)(x+\delta-\delta u,y-\delta-\delta v)\psi(u,v) \mbox{d}u\mbox{d}v.
\end{split}
\end{equation}
Next, by virtue of \eqref{part3}, we have
\begin{equation*}
\begin{split}
 \frac{U^{\delta}_{x}(x,y)}{x}&\leq \int_{[-1,1]^2} \frac{U_{x}(x+\delta-\delta u,y-\delta-\delta v)}{x+\delta-\delta u}\psi(u,v)\mbox{d}u\mbox{d}v,\\\\
 \frac{U^{\delta}_{y}(x,y)}{y}&\leq \int_{[-1,1]^2} \frac{U_{y}(x+\delta-\delta u,y-\delta-\delta v)}{y-\delta-\delta v}\psi(u,v)\mbox{d}u\mbox{d}v,
 \end{split}
 \end{equation*}
 which gives
\begin{equation}\label{aux2}
\begin{split}
U_{xx}^{\delta}(x,y)+\frac{(d-1)U_{x}^{\delta}(x,y)}{x}&\leq \int_{[-1,1]^2} L(x+\delta-\delta u,y-\delta-\delta v)\psi(u,v) \mbox{d}u\mbox{d}v,\\\\
U_{yy}^{\delta}(x,y)+\frac{(d-1)U_{y}^{\delta}(x,y)}{y} &\leq \int_{[-1,1]^2} R(x+\delta-\delta u,y-\delta-\delta v)\psi(u,v) \mbox{d}u\mbox{d}v.
\end{split}
\end{equation}
It suffices to combine \eqref{aux1} with \eqref{aux2} to obtain \eqref{conv}. The inequality $w\leq 0$ follows immediately from \eqref{part1.5}.
\end{proof}

Now we are ready to establish the submartingale inequality of Theorem \ref{mainth}.

\begin{proof}[Proof of \eqref{mainin}] Of course, we may restrict ourselves to $X\in L^p$, since otherwise there is nothing to prove.  
Fix $\delta\in (0,1/2)$ and a large positive integer $N$. Consider the stopping time $\tau=\tau^K=\inf\{t\geq 0:X_t+Y_t+B_t\geq K\}$ and introduce the process 
$Z=Z^{K,\delta}=(Z_t)_{t\geq 0}$ by setting
$$ Z_t=\begin{cases}
(2\delta+X_{\tau\wedge t},2\delta+Y_{\tau\wedge t})& \mbox{if }\tau>0,\\
(0,0) & \mbox{if }\tau=0.
\end{cases}$$
The function $U^{\delta}$ is of class $C^\infty$, so applying It\^o formula yields
\begin{equation}\label{ito}
 U^\delta(Z_t)=I_0+I_1+I_2+\frac{1}{2}I_3,
\end{equation}
where
\begin{equation*}
\begin{split}
I_0&=U^\delta(Z_0),\\
I_1&=\int_{0+}^t U^{\delta}_{x}(Z_s)\mbox{d}M_s+\int_{0+}^t U^{\delta}_{y}(Z_s)\mbox{d}N_s,\\
I_2&=\int_{0+}^t U^{\delta}_{x}(Z_s)\mbox{d}A_s+\int_{0+}^t U^{\delta}_{y}(Z_s)\mbox{d}B_s,\\
I_3&=\int_{0+}^t U^{\delta}_{xx}(Z_s)\mbox{d}[X,X]_s+2\int_{0+}^t U^{\delta}_{xy}(Z_s)\mbox{d}[X,Y]_s+\int_{0+}^t U^{\delta}_{yy}(Z_s)\mbox{d}[Y,Y]_s.
\end{split}
\end{equation*}
We may and do assume that both stochastic integrals in $I_1$ are martingales, passing to localizing sequences $(\tau_n)_{n\geq 0}$ of stopping times if necessary (and repeating the reasoning with $\tau$ replaced by $\tau\wedge \tau_n$). 
Consequently, $\E I_1=0$. To deal with $I_2$, note that by \eqref{part3} and the assumption \eqref{assumpt}, we have
$$ \int_{0+}^t U^{\delta}_{x}(Z_s)\mbox{d}A_s\leq \int_{0+}^t \frac{U^{\delta}_{x}(Z_s)}{2\delta+X_s}X_s\mbox{d}A_s\leq \int_{0+}^t \frac{U^{\delta}_{x}(Z_s)}{2\delta+X_s}\frac{d-1}{2}\mbox{d}[X,X]_s$$
and, similarly,
$$ \int_{0+}^t U^{\delta}_{y}(Z_s)\mbox{d}B_s\leq \int_{0+}^t \frac{U^{\delta}_{y}(Z_s)}{2\delta+Y_s}\frac{d-1}{2}\mbox{d}[Y,Y]_s+2\delta\int_{0+}^t \frac{U^{\delta}_{y}(Z_s)}{2\delta+Y_s}\mbox{d}B_s.$$
Hence $I_2+I_3/2\leq J_1/2+J_2$, where
\begin{equation*}
\begin{split}
J_1&=\int_{0+}^t \left[U^{\delta}_{xx}(Z_s)+\frac{(d-1)U^{\delta}_{x}(Z_s)}{2\delta+X_s}\right]\mbox{d}[X,X]_s\\
&\quad +
2\int_{0+}^t U^{\delta}_{xy}(Z_s)\mbox{d}[X,Y]_s+\int_{0+}^t \left[U^{\delta}_{yy}(Z_s)+\frac{(d-1)U^{\delta}_{y}(Z_s)}{2\delta+Y_s}\right]\mbox{d}[Y,Y]_s,\\
J_2&=\int_{0+}^t \frac{2\delta U^{\delta}_{y}(Z_s)}{2\delta+Y_s}\mbox{d}B_s.
\end{split}
\end{equation*}
Let us approximate the integrals in $J_1$ by discrete sums and use \eqref{conv} to obtain
$$ J_1\leq \int_{0+}^t w(Z_s)\mbox{d}([X,X]_s-[Y,Y]_s)\leq 0,$$
by virtue of the differential subordination and the fact that $w$ is nonpositive. We refer the reader to Wang \cite[p.~533]{W} for a detailed explanation of this step. To deal with $J_2$, note that if $Y_s\geq \sqrt{\delta}$, then
$$ \frac{2\delta U^{\delta}_{y}(Z_s)}{2\delta+Y_s}\leq 2\sqrt{\delta}\cdot \sup_{(0,K+2]\times (0,K+2]} U_{y},$$
while for $Y_s<\sqrt{\delta},$
$$ \frac{2\delta U^{\delta}_{y}(Z_s)}{2\delta+Y_s}\leq \sup_{(0,K+2]\times (0,2\delta]} U_{y}.$$
Since $\lim_{y\to 0}U_{y}(x,y)=0$ uniformly for $x\in (0,K+2]$, we see that the integrand in $J_2$ converges to $0$ as $\delta\to 0$. Hence so does $J_2$, since $B_t\leq K$ by the definition of $\tau$. 
Summarizing, if we take expectation of both sides of \eqref{ito}, we obtain
$$ \E V^\delta(Z_t)\leq \E U^\delta (Z_t)\leq \E U^\delta(Z_0)+\kappa(\delta),$$
with $\kappa(\delta)=o(1)$ as $\delta\to 0$. We have $|Z_t|\leq K+4\delta\leq K+2$ and the functions $U$, $V$ are continuous.  Thus, letting $\delta\to 0$ and applying Lebesgue's dominated convergence theorem, we get $ \E V(X_{\tau\wedge t},Y_{\tau\wedge t})\leq \E U(X_0,Y_0).$ However, as one easily checks, we have $U(x,y)\leq 0$ for $y\leq x$: this is equivalent to $z_0\geq 0$ for $p\leq 2$ and to $z_0\leq 0$ for remaining $p$. Consequently, $\E U(X_0,Y_0)\leq 0$ in view of the differential subordination and hence
$$ \E Y_{\tau\wedge t}^p\leq C_{p,d}^p\E X_{\tau\wedge t}^p\leq C_{p,d}^p||X||_p^p.$$
It suffices to let $K\to \infty$ and then $t\to\infty$ to complete the proof, by virtue of Lebesgue's monotone convergence theorem.
\end{proof} 

\begin{proof}[Proof of \eqref{orthogin} and \eqref{besselin}] This follows immediately from \eqref{mainin}. See Introduction to see how analytic martingales and stopped Bessel processes are related to nonnegative submartingales satisfying \eqref{assumpt}.
\end{proof}

\begin{proof}[The sharpness of \eqref{mainin}, \eqref{orthogin} and \eqref{besselin}.] It suffices to show that the constant $C_{p,d}$ is the best in \eqref{besselin}. We shall restrict ourselves to the stopped Bessel processes $R$, $S$ of the form \eqref{bessels}, starting from $1$. First, suppose that $p<2$. We have $z_0>0$ by Lemma \ref{lem1} (iii). Fix $a\in (0,z_0)$ and recall $\tau^a$, the stopping time defined in \eqref{deftau}. We have shown in Lemma \ref{roots} that $\tau^a\in L^{p/2}$ and that \eqref{sharpness} is valid. Therefore, letting $a\uparrow z_0$ gives the optimality of $C_{p,d}$. The same reasoning proves that \eqref{besselin} and hence also \eqref{mainin} do not hold with any finite constant when $p+d\leq 2$. If  $p=2$, then $C_{p,d}=1$, so the choice $\tau=0$ gives equality in \eqref{besselin}. Finally, suppose that $p>2$. We will switch the roles $R$ and $S$, and prove that for any $C<C_{p,d}$ there is a stopping time such that $||R_\tau||_p\geq C||S_\tau||_p$. Let $-1<b<a<z_0$. We make use of the following two-step procedure: first we let $(R,S)$ drop to the line $y=\frac{1+b}{1-b}x$ and then let it rise to the line $y=\frac{1+a}{1-a}x$. To be more precise, observe that $\mathbb{P}(\tau^b\leq 1)>0$: the process $((R_t,S_t))_{t\in [0,1]}$ reaches the line $y=\frac{1+b}{1-b}x$ with positive probability. Define $\tau=1$ if $\tau^b>1$ and
$$ \tau=\inf\left\{t>\tau^b: S_t=\frac{1-a}{1+a}R_t\right\}$$
if $\tau^b\leq 1$. By the strong Markov property and the reasoning from the proof of Lemma \ref{roots}, we have
$$ \E (S_{\tau}^p|\tau^b\leq 1)=\left(\frac{1+a}{1-a}\right)^p \E (R_{\tau}^p|\tau^b\leq 1)$$
and the expectations tend to $\infty$ as $a\to z_0$. On the other hand, we have 
$$\E S_{\tau}^p1_{\{\tau^b>1\}}=\E S_1^p1_{\{\tau^b>1\}}\leq \E S_1^p$$
and therefore
\begin{equation*}
\begin{split}
 ||R_{\tau}||_p&\geq \frac{1-a}{1+a}||S_{\tau}1_{\{\tau^b\leq 1\}}||_p\\
 &\geq \frac{1-a}{1+a}||S_{\tau}||_p-\frac{1-a}{1+a}||S_{\tau}1_{\{\tau^b> 1\}}||_p\\
 &\geq \frac{1-a}{1+a}||S_{\tau}||_p-\frac{1-a}{1+a}||S_1||_p.
 \end{split}
 \end{equation*}
 Now fix $\e>0$. If $a$ is sufficiently close to $z_0$, then
 $$ \frac{1-a}{1+a}||S_1||_p\leq \e||S_{\tau}||_p$$
 and hence
 $$ ||R_\tau||_p\geq \left(\frac{1-a}{1+a}-\e\right)||S_\tau||_p.$$
This proves the optimality of the constant $C_{p,d}$.
\end{proof}

\section{Analytic Functions on $\bC$ and Smooth functions on $\R^n$}

As discussed in the introduction, the inequality for conformal (analytic) martingales in this paper and those in \cite{BJV1, BJV2} are motivated by the martingale study of the norm of the Beurling-Ahlfors operator. However, conformal martingales have been extensively studied in the literature (see \cite{GetSha} for example) as they arise naturally from the fundamental theorem of P. L\'evy which asserts that the composition of 2-dimensional Brownian motion with an analytic function in the plane is a time change of 2-dimensional Brownian motion. We recall here the classical setting in the unit disc.  Let $D=\{z\in \bC: |z|<1\}$ be the unit disc in the plane and suppose that $F:D\to \bC$ is an analytic function with the representation $F(z)=u(z)+iv(z)$ where $u$ and $v$ are conjugate harmonic functions. If $B$ is Brownian motion in the disc and $\tau_D=\inf\{t>0: B_t\notin D\}$, then 

\begin{equation}\label{discconf}
X_t=F(B_{\tau_D\wedge t})=u(B_{\tau_D\wedge t}) +iv(B_{\tau_D\wedge t})
\end{equation}
 is a conformal martingale in $\R^2$ (identified here with $\bC$).  This follows directly from the It\^o formula); see \cite{Dur} or \cite[p.~177]{RY}.  The quadratic variation process of the martingale $X$ is given by
\begin{equation}
[X,X]_t=\int_0^{\tau_D\wedge t}|\nabla u(B_s)|^2 ds +\int_0^{\tau_D\wedge t}|\nabla v(B_s)|^2 ds=2\int_0^{\tau_D\wedge t}|\nabla u(B_s)|^2 ds,
\end{equation}
where we used the fact that $|\nabla v|=|\nabla u|$, by the Cauchy--Riemann equations.  Of course, $[X, X]$ here can also be written simply in terms of $|F'|^2$ rather than $|\nabla u|^2$.

 For any $0<p<\infty$, the classical $H_p$-norm of the analytic  function is defined by 
 \begin{equation}\label{h_p-norm}
\|F\|_{H_p}=\left[\sup_{0<r<1}\frac{1}{2\pi}\int_0^{2\pi}|F(re^{i\theta})|^p\, d\theta\right]^{1/p}. 
\end{equation}
We have the following which is an immediate consequence  of Corollary \ref{conf-cor}. 

\begin{theorem}\label{application1} If $F_1(z)=u_1(z)+iv_1(z)$ and $F_2(z)=u_2(z)+iv_2(z)$ are analytic functions in the unit disc $D$ with $|F_2(0)|\leq |F_1(0)|$ and $|F_2'(z)|\leq |F_1'(z)|$ for all $z\in D$, then for any $0<p<\infty$, 
\begin{equation}
\|F_2\|_{H_p}\leq C_{p, 2}\|F_1\|_{H_p}. 
\end{equation}
\end{theorem}

\begin{remark}The  very interesting question arises here as to whether the constant $C_{p,2}$ is optimal. Unfortunately, we have not been able to answer it, however, we strongly believe that this inequality is \emph{not} sharp, except for the trivial case $p=2$.
\end{remark}

One may replace the unit disc above with any domain in the complex plane and modify the definition of the $H_p$ norm to be with respect to the harmonic measure and obtain a similar inequality.  We leave this to the reader.  Here we state a more general inequality for smooth functions in $\R^d$ satisfying a subordination condition which arises from the submartingale condition  \eqref{assumpt}. Suppose that $D$ is an open subset of $\R^n$, where $n$ is a
fixed positive integer, and assume that $0\in D$. Let $D_0$ be a bounded subdomain of $D$ with $0\in D_0$
and $\partial D_0\subset D$. Let $\mu_{D_0}$ denote the harmonic measure
on $\partial D_0$ with respect to $0$. Consider two real-valued $C^2$ functions
$u$, $v$ on $D$, satisfying
\begin{equation}\label{A0}
|v(0)|\leq |u(0)|.
\end{equation}
Following \cite{B0.9}, $v$ is
differentially subordinate to $u$ if
\begin{equation}\label{A1} 
|\nabla v(x)|\leq |\nabla u(x)|\quad \mbox{ for }x\in D.
\end{equation}
Let us assume further that there is $d>1$ such that
\begin{equation}\label{A2}
u(x)\Delta u(x)\geq (d-1)|\nabla u(x)|^2 \quad\mbox{and}\quad  v(x)\Delta v(x)\leq (d-1)|\nabla v(x)|^2  
\end{equation}
for all $x\in D$. In what follows, 
$$ ||u||_p=\sup\left[\int_{\partial D_0}|u(x)|^p \mu_{D_0}(\mbox{d}x)\right]^{1/p},$$
where the supremum is taken over all $D_0$ as above. 

The condition \eqref{A2} appears naturally while studying a Stein-Weiss system of harmonic functions. Let $u_j$, $j=0,\,1,\,2,\,\ldots,\,n$, be harmonic functions given on an open subset of $\R\times \R^n$,  taking values in a certain separable Hilbert space. Assume that they satisfy the generalized Cauchy-Riemann equations
$$ \sum_{j=0}^n \frac{\partial u_j}{\partial x_j}=0\qquad \mbox{and}\qquad \frac{\partial u_j}{\partial x_k}=\frac{\partial u_k}{\partial x_j}$$
for all $j,\,k\in \{0,\,1,\,2,\,\ldots,\,n\}$. Let $F$ stand for the vector $(u_0,\,u_1,\,\ldots,\,u_n)$ and fix $q>(n-1)/n$. Then the function $u=|F|^q$ satisfies the left inequality in \eqref{A2} with $d=2-\frac{n-1}{nq}>1.$ To see this, we easily compute that
$$ |\nabla |F|^q|^2=q^2|F|^{2q-4}\sum_{j=0}^n\left(\frac{\partial F}{\partial x_j}\cdot F\right)^2$$
and
$$ \Delta |F|^q=q|F|^{q-4}\left((q-2)\sum_{j=0}^n \left(\frac{\partial F}{\partial x_j}\cdot F\right)^2+|F|^2|\nabla F|^2\right).$$
It suffices to apply the estimate
$$ \sum_{j=0}^n\left(\frac{\partial F}{\partial x_j}\cdot F\right)^2 \leq \frac{n}{n+1}|F|^2|\nabla F|^2$$
(see page 219 in Stein \cite{S}) to obtain
$$ |F|^q \Delta |F|^q\geq \left(1-\frac{n-1}{nq}\right)|\nabla |F|^q|^2. $$

\begin{theorem}
If $u$, $v$ are nonnegative subharmonic functions satisfying \eqref{A0}, \eqref{A1} and \eqref{A2}, then for any $(2-d)_+<p<\infty$,
$$ ||v||_p\leq C_{p,d} ||u||_p.$$
\end{theorem}
\begin{proof}
Pick any $D_0$ as above. Obviously, we will be done if we show that
\begin{equation}\label{want}
\left[\int_{\partial D_0}|v(x)|^p \mu_{D_0}(\mbox{d}x)\right]^{1/p}\leq C_{p,d}\left[\int_{\partial D_0}|u(x)|^p \mu_{D_0}(\mbox{d}x)\right]^{1/p}.
\end{equation}
Let $B=(B_t)_{t\geq 0}$ be a Brownian motion in $\R^n$, starting at $0$, and let $\tau_{D_0}=\inf\{t\geq 0:B_t\notin D_0\}$. 
Define $X_t=u(B_{\tau_{D_0}\wedge t})$ and $Y_t=v(B_{\tau_{D_0}\wedge t})$ for $t\geq 0$.  From the It\^o formula we see that $X$, $Y$ are nonnegative submartingales with the corresponding Doob-Meyer decompositions given by
$$ X_t=u(0)+\int_{0+}^{\tau_{D_0}\wedge t} \nabla u(B_s)\cdot\mbox{d}B_s+\frac{1}{2}\int_{0+}^{\tau_{D_0}\wedge t} \Delta u(B_s) ds,$$
$$ Y_t=v(0)+\int_{0+}^{\tau_{D_0}\wedge t} \nabla v(B_s)\cdot\mbox{d}B_s+\frac{1}{2}\int_{0+}^{\tau_{D_0}\wedge t} \Delta v(B_s) ds.$$
Therefore, the assumptions \eqref{A0} and \eqref{A1} imply that $Y$ is differentially subordinate to $X$, while \eqref{A2} yields \eqref{assumpt}. Consequently, by \eqref{mainin}, we have
$$ ||v(B_{\tau_{D_0}})||_p\leq C_{p,d} ||u(B_{\tau_{D_0}})||_p,$$
which is equivalent to \eqref{want} since the distribution of $B_{\tau_{D_0}}$ is $\mu_{D_0}$. The proof is complete.
\end{proof}

\section*{Acknowledgment} The results were obtained when the second author was visiting Purdue University.  We also thank the anonymous referee for useful comments.


\begin{thebibliography}{99}

\bibitem{Ab}{M. Abramowitz and I. A. Stegun, {\it Handbook of mathematical functions with formulas,
graphs and mathematical tables}, Reprint of the 1972 edition, Dover Publications, Inc.,
New York, 1992.}



\bibitem{AstIwaMar} K. Astala, T. Iwaniec and G. Martin, \emph{Elliptic Partial Differential Equations and Quasiconformal Mappings in the Plane,} Princeton University Press, 2009. 

\bibitem{AstIwaPraSak} K. Astala, T. Iwaniec, I. Prause, E. Saksman, \emph{ Burkholder integrals, Morrey's problem and quasiconformal mappings}, Journal of the American Math. Soc., S 0894-0347(2011)00718-2
Electronically Publication, October 6, 2011.  


\bibitem{Ban1} R. Ba\~nuelos, \emph{The foundational inequalities of D. L. Burkholder and some of their ramifications}, 
To appear, Illinois Journal of Mathematics, Volume in honor of D.L. Burkholder.

\bibitem{BanJan} R. Ba{\~n}uelos and P. Janakiraman, \emph{$L\sp p$-bounds for the {B}eurling-{A}hlfors transform}, Trans. Amer. Math. Soc.,  \textbf{360} ({2008}), no. {7}, {3603--3612}. 


\bibitem{BMH}{R. Ba\~nuelos and P. J. M\'endez-Hernandez, \emph{Space-time Brownian motion and the Beurling-Ahlfors transform}, Indiana Univ. Math. J. \textbf{52} (2003), no. 4, 981--990.}
\bibitem{BW0}{R. Ba\~nuelos and G. Wang, \emph{Sharp inequalities for
martingales with applications to the Beurling-Ahlfors and Riesz
transforms}, Duke Math. J. \textbf{80} (1995), no. 3, 575--600.}
\bibitem{BW}{R. Ba\~nuelos and G. Wang, \emph{Orthogonal martingales under differential subordination and applications to Riesz transforms}, Illinois J. Math. \textbf{40} No. 4 (1996), pp. 678--691.} 

\bibitem{BJV1}{A. Borichev, P. Janakiraman and A. Volberg, \emph{Subordination by orthogonal martingales in $L^{p}$ and zeros of Laguerre polynomials}, arXiv:1012.0943.} 
\bibitem{BJV2}{A. Borichev, P. Janakiraman and A. Volberg, \emph{On Burkholder function for orthogonal martingales and zeros of Legendre polynomials}, Amer. Jour. Math. {\bf (to appear)}} 
\bibitem{B0}{D. L. Burkholder, \emph{Boundary value problems and sharp inequalities for martingale transforms}, 
Ann. Probab. \textbf{12} (1984), 647--702.}

\bibitem{Bur43}  D.L. Burkholder, \emph{A sharp and strict $L^p$-inequality for stochastic integrals,} Ann. Probab. {\bf 15} (1987 ), 268-273.

\bibitem{B-1.1}{D. L. Burkholder, {\it A proof of Pe\l czy\'nski's conjecture for the Haar system}, Studia Math. \textbf{91} (1988), 79--83.}
\bibitem{B22}{D. L. Burkholder, {\it Sharp inequalities for martingales and stochastic integrals}, Ast\'erisque \textbf{157--158} (1988), 75--94.}
\bibitem{B0.9}{{D. L. Burkholder}, {\it Differential subordination of harmonic functions and martingales}, Harmonic Analysis and Partial Differential Equations (El Escorial, 1987), Lecture Notes in Mathematics 1384 (1989), 1--23.}
\bibitem{B1}{D. L. Burkholder, \emph{Explorations in martingale theory and its applications}, \'Ecole d'Et\'e de Probabilit\'es de Saint-Flour XIX---1989, pp. 1--66, Lecture Notes in Math., 1464, Springer, Berlin, 1991.}
\bibitem{Dav}{B. Davis, {\it On the $L^p$ norms of stochastic integrals and other martingales}, Duke Math. J. \textbf{43} (1976), 697--704.}

\bibitem{Deb}{R. D. DeBlassie, {\it Stopping times of Bessel processes}, Ann. Probab. \textbf{15} (1987), 1044--1051.}
\bibitem{DM}{C. Dellacherie and P. A. Meyer, {\it Probabilities and Potential B: Theory of martingales}, North Holland, Amsterdam, 1982.}

\bibitem{Dur} R.Durrett, \emph{Brownian Motion and Martingales
in Analysis},
Wadsworth,  Belmont, CA, 1984.


\bibitem{G}{T. W. Gamelin, \emph{Uniform algebras and Jensen measures}, Cambridge University Press, London, 1978.}
\bibitem{GME}{S. Geiss, S. Montgomery-Smith and E. Saksman, \emph{On singular integral and martingale transforms}, Trans. Amer. Math. Soc. \textbf{362} No. 2 (2010), 553--575.}

\bibitem{GetSha} R.K.Getoor and M.J. Sharpe, \emph{Conformal martingales,} Invent. Math. {\bf 16} (1972), 271--308.

\bibitem{Iwa} T. Iwaniec, \emph{Extremal inequalities in Sobolev spaces and quasiconformal mappings,} Z. Anal. 
Anwendungen {\bf 1} (1982), 1--16.

\bibitem{Iwa3} T. Iwaniec, \emph{Nonlinear Cauchy-Riemann operators in ${\Bbb R}\sp n$,} Trans. Amer. Math.
Soc. {\bf 354} (2002), 1961--1995.

\bibitem{JanVol1} P. Janakiraman and A. Volberg, \emph{Subordination by orthogonal martingales in $L^p$, $1<p\leq 2$.} preprint. 

\bibitem{Jan1}{P. Janakiraman, \emph{Best weak-type $(p,p)$ constants, $1\leq p\leq 2$, for orthogonal harmonic functions and martingales}, Illinois J. Math. \textbf{48} no. 3 (2004),  909--924.}


\bibitem{NazTre1} F.L. Nazarov and S.R. Treil, \emph{The hunt for a Bellman function: applications to estimates for singular integral operators and to other classical problems of harmonic analysis,} St. Petersburg Math. J. {\bf 8} (1997), 721--824. 

\bibitem{NazTreVol} F. L. Nazarov, S. R. Treil and A. Volberg, \emph{Bellman function in stochastic optimal control and harmonic analysis (how our Bellman function got its name),} Oper. Theory: Adv. Appl. {\bf 129} (2001),  393-424.

\bibitem{NazVol} F. L. Nazarov and A. Volberg, \emph{Heat extension of the {B}eurling operator and estimates for its norm}, {Rossi\u\i skaya Akademiya Nauk. Algebra i Analiz}, \textbf{15}, No. 4 (2003), 142--158.

 \bibitem{O1}{A. Os\c ekowski}, \emph{Sharp inequalities for differentially subordinate harmonic functions and martingales,} to appear in Canadian Mathematical Bulletin

 \bibitem{O2}{A. Os\c ekowski}, \emph{A sharp weak-type bound for It\^o processes and subharmonic functions,} to appear in Kyoto Journal of Mathematics

\bibitem{O3}{A. Os\c ekowski}, \emph{Maximal inequalities for continuous martingales and their differential subordinates}, Proc. Amer. Math. Soc. \textbf{139} (2011), pp. 721--734.

 

\bibitem{O4}{A. Os\c ekowski}, \emph{Sharp and strict $L^p$-inequalities for Hilbert-space-valued orthogonal martingales,} Electronic Journal of Probability {\bf 16} (2011), pp. 531--551.

\bibitem{O5}{A. Os\c ekowski}, \emph{Sharp LlogL inequality for Differentially Subordinated Martingales,} Illinois Journal of Mathematics {\bf 52,} Vol. 3 (2008), pp. 745--756.
	

\bibitem{Ped}{J. L. Pedersen, {\it Best Bounds in Doob's Maximal Inequality for Bessel Processes}, J. Multivariate Anal. \textbf{75} (2000), 36--46.}
\bibitem{P}{S. K. Pichorides, {\it On the best values of the constants in the theorems of M. Riesz, Zygmund and Kolmogorov}, Studia Math. \textbf{44} (1972), 165--179.}
\bibitem{RY}{D. Revuz and M. Yor, \emph{Continuous martingales and Brownian Motion}, 3rd edition, Springer-Verlag, Berlin, 1999.}
\bibitem{S}{E. M. Stein, {\it Singular Integrals and Differentiability Properties of Functions}, Princeton University Press, 1970.}

\bibitem{VasVol} V. Vasyunin and A. Volberg, \emph{Bellman functions technique in harmonic
analysis}, (sashavolberg.wordpress.com). 

\bibitem{W}{G. Wang, \emph{Differential subordination and strong
differential subordination for continuous-time martingales and
related sharp inequalities}, Ann. Probab. \textbf{23} no. 2 (1995), 522--551.}
\end{thebibliography}
\end{document}